\documentclass[11pt]{amsart}
\usepackage{color}
\usepackage{amsmath,amsthm,amsfonts,amssymb}

\begin{document}
\title{Nonlocal $p$-Kirchhoff equations with singular and critical nonlinearity terms}
\thanks{Math. Subj. Classif. (2010):  34B15, 37C25, 35R20}
\date{}
\maketitle
\begingroup\small
\begin{center}
{\sc Abdeljabbar  Ghanmi} \\
Facult\'e des Sciences de Tunis, LR10ES09 Mod\'elisation Mat\'ematique, Analyse Harmonique et
Th\'eorie du Potentiel, Universit\'e de Tunis El Manar, Tunis 2092, Tunisie.
 email: abdeljabbar.ghanmi@lamsin.rnu.tn
 \\[10pt]
{\sc Mouna Kratou} \\
  College of Sciences, Imam Abdulrahman Bin Faisal University,\\
 31441  Dammam, Kingdom of Saudi Arabia \\
 email: mmkratou@iau.edu.sa
\\[10pt]
{\sc Kamel  Saoudi} \\
  College of Sciences, Imam Abdulrahman Bin Faisal University,\\
 31441  Dammam, Kingdom of Saudi Arabia \\ 
 email: kmsaoudi@iau.edu.sa
 \\[10pt]
 {\sc Du\v{s}an D. Repov\v{s}} \\
 Faculty of Education and Faculty of Mathematics and Physics, University of Ljubljana \&
 Institute of Mathematics, Physics and Mechanics,  \\
1000 Ljubljana, Slovenia \\
 email: dusan.repovs@guest.arnes.si
\end{center}
\endgroup
\numberwithin{equation}{section}
\allowdisplaybreaks
\newtheorem{theorem}{Theorem}[section]
\newtheorem{proposition}{Proposition}[section]
\newtheorem{gindextheorem}{General Index Theorem}[section]
\newtheorem{indextheorem} {Index Theorem}[section]
\newtheorem{standardbasis}{Standard Basis}[section]
\newtheorem{generators} {Generators}[section]
\newtheorem{lemma} {Lemma}[section]
\newtheorem{corollary}{Corollary}[section]
\newtheorem{example}{Example}[section]
\newtheorem{examples}{Examples}[section]
\newtheorem{exercise}{Exercise}[section]
\newtheorem{remark}{Remark}[section]
\newtheorem{remarks}{Remarks}[section]
\newtheorem{definition} {{Definition}}[section]
\newtheorem{definitions}{Definitions}[section]
\newtheorem{notation}{Notation}[section]
\newtheorem{notations}{Notations}[section]
\newtheorem{defnot}{Definitions and Notations}[section]
\def\H{{\mathbb H}}
\def\N{{\mathbb N}}
\def\R{{\mathbb R}}
\newcommand{\be} {\begin{equation}}
\newcommand{\ee} {\end{equation}}
\newcommand{\bea} {\begin{eqnarray}}
\newcommand{\eea} {\end{eqnarray}}
\newcommand{\Bea} {\begin{eqnarray*}}
\newcommand{\Eea} {\end{eqnarray*}}
\newcommand{\p} {\partial}
\newcommand{\ov} {\over}
\newcommand{\al} {\alpha}
\newcommand{\ba} {\beta}
\newcommand{\de} {\delta}
\newcommand{\ga} {\gamma}
\newcommand{\Ga} {\Gamma}
\newcommand{\Om} {\Omega}
\newcommand{\om} {\omega}
\newcommand{\De} {\Delta}
\newcommand{\la} {\lambda}
\newcommand{\si} {\sigma}
\newcommand{\Si} {\Sigma}
\newcommand{\La} {\Lambda} 
\newcommand{\no} {\nonumber}
\newcommand{\noi} {\noindent}
\newcommand{\lab} {\label}
\newcommand{\na} {\nabla}
\newcommand{\vp} {\varphi}
\newcommand{\var} {\varepsilon}
\newcommand{\RR}{{\mathbb R}}
\newcommand{\CC}{{\mathbb C}}
\newcommand{\NN}{{\mathbb N}}
\newcommand{\ZZ}{{\mathbb Z}}
\renewcommand{\SS}{{\mathbb S}}
\newcommand{\supp}{\mathop{\rm {supp}}\limits}    
\newcommand{\esssup}{\mathop{\rm {ess\,sup}}\limits}    
\newcommand{\essinf}{\mathop{\rm {ess\,inf}}\limits}   
\newcommand{\weaklim}{\mathop{\rm {weak-lim}}\limits}
\newcommand{\wstarlim}{\mathop{\rm {weak^\ast-lim}}\limits}
\newcommand{\RE}{\Re {\mathfrak e}}   
\newcommand{\IM}{\Im {\mathfrak m}}    
\renewcommand{\colon}{:\,}
\newcommand{\eps}{\varepsilon}
\newcommand{\half}{\textstyle\frac12}
\newcommand{\Takac}{Tak\'a\v{c}}
\newcommand{\eqdef}{\stackrel{{\rm {def}}}{=}}  
\newcommand{\wstarconverge}{\stackrel{*}{\rightharpoonup}}
\newcommand{\Div}{\nabla\cdot}
\newcommand{\Curl}{\nabla\times}
\newcommand{\Meas}{\mathop{\mathrm{meas}}}
\newcommand{\Int}{\mathop{\mathrm{Int}}}
\newcommand{\Clos}{\mathop{\mathrm{Clos}}}
\newcommand{\Lin}{\mathop{\mathrm{lin}}}
\newcommand{\Dist}{\mathop{\mathrm{dist}}}
\newcommand{\Square}{$\sqcap$\hskip -1.5ex $\sqcup$}
\newcommand{\Blacksquare}{\vrule height 1.7ex width 2.0ex depth 0.2ex }

\smallskip
\begin{abstract} The objective  of this work is to investigate  a nonlocal problem involving  singular and critical nonlinearities:
\begin{equation*}
\left\{
  \begin{array}{ll}
([u]_{s,p}^p)^{\sigma-1}(-\Delta)^s_p u
= \frac{\lambda}{u^{\gamma}}+u^{ p_s^{*}-1 }\quad \text{in }\Omega,\\
u>0,\;\;\;\;\quad \text{in }\Omega,\\
u=0,\;\;\;\;\quad \text{in }\mathbb{R}^{N}\setminus \Omega,
  \end{array}
\right.
\end{equation*}
 where  $\Omega$ is a bounded domain in $\mathbb{R}^N$ with the smooth boundary
 $\partial \Omega$,   $0 < s< 1<p<\infty$,  $N> sp$, $1<\sigma<p^*_s/p,$ with $p_s^{*}=\frac{Np}{N-ps},$
 $ (-  \Delta )_p^s$ is the nonlocal $p$-Laplace operator  and $[u]_{s,p}$ is the Gagliardo $p$-seminorm.
 We combine some variational techniques with a truncation argument in order to  show the existence
and the multiplicity of positive solutions to the above problem.
\end{abstract}
\vfill\eject

\section{Introduction}
In this paper, we shall consider the following singular critical nonlocal problem:
\begin{equation} \label{eq}
\left\{
  \begin{array}{ll}
([u]_{s,p}^p)^{\sigma-1}(-\Delta)^s_p u
= \frac{\lambda}{u^{\gamma}}+u^{ p_s^{*}-1 }\quad \text{in }\Omega,\\
u>0,\;\;\;\;\quad \text{in }\Omega,\\
u=0,\;\;\;\;\quad \text{in }\mathbb{R}^{N}\setminus \Omega,
  \end{array}
\right.
\end{equation}

\noindent
where  $\Omega$ is a bounded domain in $\mathbb{R}^N$ with a smooth boundary
 $\partial \Omega$, $0 < s< 1<p<\infty$,  $N> sp$, $1<\sigma<p^*_s/p,$  with $p_s^{*}=\frac{Np}{N-ps},$
$ (-\Delta )_p^s$ is a nonlocal  operator  defined by
$$
(-  \Delta )_p^s u(x) :=
2 \lim_{\epsilon\to 0} \int_{\Omega\backslash B_{\epsilon}(x)}
\frac{|u(x)-u(y)|^{p-2}(u(x)-u(y))}{|x-y|^{N
 +ps}}\mathrm{d}y,\quad x\in \Omega,
$$
where
 $B_\epsilon(x) := \{y \in \Omega : |x - y| < \epsilon\}$, and $[u]_{s,p}$ is the Gagliardo $p$-seminorm given by
 $$[u]_{s,p}^p:=\iint_{\mathbb{R}^{2N}}\frac{|u(x)-u(y)|^p}{|x-y|^{N+ps}}\,dx\,dy.$$

Problems of this type describe diffusion processes in  heterogeneous or complex medium (anomalous diffusion) due to random displacements executed by jumpers that are able to walk to neighbouring  nearby
sites. These problems are  also due to  excursions to remote sites by way of L\'evy flights,  they can be used in modelling
turbulence, chaotic dynamics, plasma physics and financial dynamics. For more details, see  \cite{Ap, CoTa} and  references therein.

For $p = 2$, problem \eqref{eq} has been investigated by many authors in order to show the existence and the multiplicity of solutions. For further details, one can refer the reader to  \cite{DaHaSa1,DaHaSa,Fi, KrSaSa, LeLiTa, LiKeLeTa, LiZhLiTa, LiSu, LiTaLiWu, Sa2} and the references therein.

For $s=1$, the  local setting case has been extensively investigated in the recent past. The existence, the uniqueness, the multiplicity of weak solutions  and regularity of solutions  have been studied
in \cite{CoPa,CrRaTa,DJP,G3, GiSa,GhRa,heana,Ra,SaKr,Sa3,shenana} and the references therein.

Motivated by the previous results,  and the work of {\sc Fiscella}  \cite{Fi},   who
 established the existence  and the  multiplicity of  positive solutions using    some variational methods combined with an appropriate truncation. The aim of this work is to extend the multiplicity results  to a more general non-local  problem. More precisely, we shall establish the following result. 
\begin{theorem}\label{thm} Suppose that the parameters in problem \eqref{eq} satisfy the following two conditions
$$0<1-\gamma<1< p\sigma<p_s^\ast
\ \
\mbox{ and}
 \ \
1<\sigma<p^*_s/p.$$
 Then there exists a
 parameter $\lambda_{0}>0$ such that
 for every $\lambda\in (0,\lambda_{0})$, problem
\eqref{eq} has at least two positive solutions.
\end{theorem}

\section{Preliminaries}\label{Main}
 This section is devoted to basic definitions, notations, and function spaces  that will be used in the forthcoming sections.  For  for the other background material   we
refer the reader to
 \cite{PRR,RR}.
 We begin by defining the fractional Sobolev space
$$W^{s,p}(\mathbb{R}^{N}):=\left\{u\in L^{p}(\mathbb{R}^{N}):u\mbox{ measurable }, |u|_{s,p}<\infty\right\},$$
 with the Gagliardo norm
$$||u||_{s,p}:=\left(||u||_{p}^{p}+|u|_{s,p}^{p}\right)^{\frac{1}{p}}.$$
Denote
$$\displaystyle\mathcal{Q}:=\mathbb{R}^{2N}\setminus \left((\mathbb{R}^{N}\setminus \Omega)\times (\mathbb{R}^{N}\setminus \Omega)\right)$$
and   define the space
\begin{eqnarray*}
X:=\left\{u:\mathbb{R}^{N}\rightarrow \mathbb{R}\;\mbox{Lebesgue measurable}:\,u\mid_{\Omega}\in L^{p}(\Omega)\,\mbox{and}\,\frac{|u(x)-u(y)|^{p}}{|x-y|^{N+s p}}\in L^{p}(\mathcal{Q})\right\}
\end{eqnarray*}
with the norm
\begin{align*}
 \|u\|_X := \|u\|_{L^p(\Omega)} +  \Big( \int_{\mathcal{Q}}\frac{|u(x)-u(y)|^{p}}{|x-y|^{N+sp}}dx
dy\Big)^{1/p}.
\end{align*}
Throughout this paper,  we  shall consider the space
$$\displaystyle X_0:=\left\{u\in X: u=0 \;a.e. \;\mbox{in}\;\mathbb{R}^{n}\setminus \Omega\right\},$$
with the norm
\begin{align*}
 \|u\| :=   \Big( \int_{\mathcal{Q}}\frac{|u(x)-u(y)|^{p}}{|x-y|^{N+sp}}dx
dy\Big)^{1/p}
\end{align*}
and the scalar product
\begin{align*}
\langle u, \varphi\rangle_{X_0}:=\iint_{\mathbb{R}^{2N}} & & \frac{|u(x)-u(y)|^{p-2}(u(x)-u(y))
 (\varphi(x)-\varphi(y))}{|x-y|^{N +ps}}\mathrm{d}x\mathrm{d}y.
\end{align*}

We define a weak solution to
 problem \eqref{eq} as follows:
\begin{definition}
We say that $u\in X_0$ is a weak solution of problem \eqref{eq} if for all $\varphi\in X_0,$ one has
\begin{eqnarray}\label{weaksolution}
 ([u]_{s,p}^p)^{\sigma-1}\iint_{\mathbb{R}^{2N}} & & \frac{|u(x)-u(y)|^{p-2}(u(x)-u(y))
 (\varphi(x)-\varphi(y))}{|x-y|^{N +ps}}\mathrm{d}x\mathrm{d}y\nonumber \\
    &=& \lambda\int_{\Omega}(u^+)^{-\gamma} \varphi\mathrm{d}x
 +\int_{\Omega}(u^+)^{p_s^{*}-1} \varphi \mathrm{d}x.
\end{eqnarray}
\end{definition}
In order to find solutions of problem \eqref{eq},  we shall use the variational approach.  More precisely, we shall find two distinct critical points of the energy functional $J_{\lambda} : X_{0} \to (-\infty, \infty]$  defined  by
 \begin{equation}\label{foctional}
  J_{\lambda}(u):=\frac{1}{p\sigma}\|u\|^{p\sigma}-\frac{\lambda}{1-\gamma}\int_{\Omega}(u^+)^{1-\gamma} \mathrm{d}x
 -\frac{1}{p_s^{*}}\int_{\Omega}(u^+)^{p_s^{*}}  \mathrm{d}x.
\end{equation}

Now, we prove the following result.
\begin{lemma} \label{lem1} There exist $\rho\in(0,1]$,  $\lambda_{1}$ and $\alpha>0$  such that for every $\lambda\in(0,\lambda_{1}]$,  we have
$$J_{\lambda}(u)\geq\alpha\;\;\ \mbox{for all} \ u\in X_{0}\mbox{ with } \|u\|=\rho.$$
Moreover, the following holds $$m_{\lambda}:=\displaystyle\inf\left\{J_{\lambda}(u): u\in \overline{B}_{\rho}\right\}<0,$$
where $\overline{B}_{\rho}:=\left\{u\in X_{0}:\|u\|\leq\rho\right\}$.
\end{lemma}
\begin{proof}
  Let $\lambda>0$. Then
  by virtue of the H\"older inequality and the Sobolev embedding theorem, we get for any $u\in X_{0}$
\begin{eqnarray*}
  \int_{\Omega} u^{1-\gamma}dx &\leq&  |\Omega|^{\frac{p^\ast_s-1+\gamma}{p^\ast_s}}\|u\|_{p^\ast_s}^{1-\gamma} \\
    &\leq&   C\|u\|^{1-\gamma}.
\end{eqnarray*}
So  from the Sobolev embedding, we obtain
\begin{eqnarray*}
  J_{\lambda}(u) &=&  \frac{1}{p\sigma}\|u\|^{p\sigma}-\frac{\lambda}{1-\gamma}\int_{\Omega}u^{1-\gamma} \mathrm{d}x
 -\frac{1}{p_s^{*}}\int_{\Omega}u^{p_s^{*}}  \mathrm{d}x \\
    &\geq& \frac{1}{p\sigma}\|u\|^{p\theta}-\frac{C\lambda}{1-\gamma}\|u\|^{1-\gamma}-\frac{c_{1}}{p^\ast_s}\|u\|^{p^\ast_s}   \\
    &=&   \|u\|^{1-\gamma}\left(\varphi(\|u\|)-\frac{C\lambda}{1-\gamma}\right)
\end{eqnarray*}
where $\varphi(t)=\frac{1}{p\sigma}t^{p\sigma-1+\gamma}-\frac{c_{1}}{p^\ast_s}t^{p^\ast_s-1+\gamma}$. Since $1-\gamma<1<p\sigma<p^\ast_s$, we find $\rho\in(0,1)$ sufficiently small
and  satisfying
\begin{equation}\label{rho}
\displaystyle\max_{0<t<1}\varphi(t)=\varphi(\rho).
\end{equation}
Put
\begin{equation}\label{lam0}
 \lambda_{1}:=\frac{(1-\gamma)\varphi(\rho)}{2C}.
\end{equation}
Thus, $\ \mbox{for all} \  u\in X_{0}$ with $\|u\|=\rho$ and $\ \mbox{all} \ \lambda\leq \lambda_{1}$, one has
$$J_{\lambda}(u)\geq\frac{C \rho^{1-\gamma}}{1-\gamma}(2\lambda_{1}-\lambda)>\frac{C \rho^{1-\gamma}}{1-\gamma}\lambda_{1}=\alpha>0.$$
Moreover, since $1-\gamma<1<p\sigma<p^\ast_s$, it follows that  for $u\in X_{0}$ with $u^+\not\equiv 0$ and for $t\in(0,1)$ sufficiently small, one has
\begin{eqnarray*}
 J_\lambda(tu) &=&\frac{ t^{p\sigma}}{p\sigma}\|u\|^{p\sigma}-\frac{\lambda t^{1-\gamma}}{1-\gamma}\int_{\Omega}(u^{+})^{1-\gamma} \mathrm{d}x
 -\frac{t^{p_s^{*}}}{p_s^{*}}\int_{\Omega}(u^{+})u^{p_s^{*}}  \mathrm{d}x \\
   &<&  0.
\end{eqnarray*}
\end{proof}

\begin{lemma} For every $\lambda\in(0,\lambda_{1}]$, problem \eqref{eq} has a positive solution $u_{\lambda}\in X_0$ with $J_{\lambda}(u_{\lambda})<0$.
\end{lemma}
\begin{proof}
  Let $\rho$ and $\lambda_{1}$ be the  constants given respectively by \eqref{rho} and \eqref{lam0}. Let $\{u_{k}\}\subset \overline{B}_{\rho}$ be a minimizing sequence for $m_\lambda,$ i.e.
   $$\displaystyle\lim_{k\rightarrow \infty}J_\lambda(u_k)=m_\lambda.$$
As $\{u_{k}\}$ is bounded, for any $1\leq r< p_s^{*}$, one has
\begin{equation}\label{lim}
\left\{
  \begin{array}{ll}
    u_{k}\rightharpoonup u_{\lambda}\;\mbox{ weakly in } X_{0}, \\
   u_{k}\rightharpoonup u_{\lambda}\;\mbox{ weakly in } L^{p_s^{*}}(\Omega), \\
    u_{k}\rightarrow u_{\lambda}\;\mbox{ strongly in } L^{r}(\Omega),  \\
    u_{k}\rightarrow u_{\lambda} \mbox{ a.e. in  } \Omega.
  \end{array}
\right.
\end{equation}
By the H\"older inequality, we get for all integers $k$,
\begin{eqnarray}\label{34}
  \left|\int_{\Omega}(u_{k}^{+})^{1-\gamma}dx-\int_{\Omega}(u_{\lambda}^{+})^{1-\gamma}dx\right|
  &\leq& \int_{\Omega}|u_{k}^{+})-u_{\lambda}^{+}|^{1-\gamma}dx\nonumber\\
  &\leq& |\Omega|^{\frac{p-1+\gamma}{p}}\|u_{k}^{+})-u_{\lambda}^{+}\|_{p}^{1-\gamma}.
\end{eqnarray}
Combining \eqref{lim} and \eqref{34}, we obtain
\begin{equation}\label{3.4}
  \displaystyle\lim_{k\rightarrow \infty}\int_{\Omega}(u_{k}^{+})^{1-\gamma}dx=\int_{\Omega}(u_{\lambda}^{+})^{1-\gamma}dx.
\end{equation}
Put $\widetilde{u}_{k}:=u_{k}-u_{\lambda}$. Then, by  invoking
the Brezis-Lieb Lemma \cite{lieb}, we obtain
\begin{equation}\label{limlieb}
  \displaystyle\lim_{k\rightarrow \infty}\|u_k\|^{p}-\|\widetilde{u}_k\|^{p}=\|u_\lambda\|^{p}\;\mbox{ and } \displaystyle\lim_{k\rightarrow \infty}\|u_k\|^{p_s^\ast}_{p_s^\ast}-\|\widetilde{u}_k\|^{p_s^\ast}_{p_s^\ast}=\|u_\lambda\|^{p_s^\ast}_{p_s^\ast}.
\end{equation}
Since $\{u_{k}\}\subset \overline{B}_{\rho}$,  it follows that \eqref{limlieb} implies that for $k$ large enough, $\widetilde{u}_k\in \overline{B}_{\rho}$. So, from Lemma \ref{lem1},  we deduce  that for all $u\in X_{0}$ with $\|u\|=\rho$,
$$\frac{1}{p\sigma}\|u\|^{p\sigma} -\frac{1}{p_s^{*}}\int_{\Omega}u^{p_s^{*}}  \mathrm{d}x \geq \alpha >0,$$
that is, if $\rho\leq 1$ and $k$ large enough,
\begin{equation}\label{3.6}
\frac{1}{p\sigma}\|\widetilde{u}_{k}\|^{p\sigma} -\frac{1}{p_s^{*}}\int_{\Omega}\widetilde{u}_{k}^{p_s^{*}}  \mathrm{d}x >0,
\end{equation}
since $\{u_k\}$ is a minimizing sequence. Hence, by combining \eqref{3.4}-\eqref{3.6}, we obtain for $k$ large enough,
\begin{eqnarray*}
  m_{\lambda} &=& J_{\lambda}(u_k)+\circ (1) \\
   &=& \frac{1}{p\sigma}\|\widetilde{u}_{k}+u_{\lambda}\|^{p\sigma}-\frac{\lambda}{1-\gamma}\int_{\Omega}((\widetilde{u}_{k}+u_{\lambda})^{+})^{1-\gamma} \mathrm{d}x -\frac{1}{p_s^{*}}\int_{\Omega}((\widetilde{u}_{k}+u_{\lambda})^{+})^{p_s^{*}}  \mathrm{d}x +\circ(1)\\
   &\geq&  \displaystyle \frac{1}{p\sigma}\|\widetilde{u}_{k}\|^{p\sigma}+\frac{1}{p\sigma}\|u_{\lambda}\|^{p\sigma}-\frac{\lambda}{1-\gamma}\int_{\Omega}(u_{\lambda}^{+})^{1-\gamma} \mathrm{d}x -\frac{1}{p_s^{*}}\int_{\Omega}(\widetilde{u}_{k}^{+})^{p_s^{*}}-\frac{1}{p_s^{*}}\int_{\Omega}(u_{\lambda}^{+})^{p_s^{*}}  \mathrm{d}x +\circ(1)\\
   &\geq&  J_{\lambda}(u_\lambda)+\frac{1}{p\sigma} \|\widetilde{u}_{k}\|^{p\sigma}-\frac{1}{p_s^{*}}\int_{\Omega}(\widetilde{u}_{k}^{+})^{p_s^{*}}+\circ(1)\\
&\geq& J_{\lambda}(u_\lambda)+\circ(1)\\
&\geq& m_{\lambda},
\end{eqnarray*}
Hence, $J_{\lambda}(u_\lambda)=m_{\lambda}<0$.

Now, let us prove that $u_\lambda$ is a positive solution to problem \eqref{eq}.
  Our proof
uses  similar techniques as \cite{GhSa1}.
 Consider
 $\phi\in X_0$ and $0<\epsilon<1.$ Let
$\Psi\in X_0$ be defined by
$
\Psi:=(u_{\lambda}+\epsilon\phi)^{+}$
with $(u_{\lambda}+\epsilon\phi)^{+}:=\max \{u_{\lambda}+\epsilon\phi,0\}.$  Let $\Omega_\epsilon:=\{u_{\lambda}+\epsilon\phi\leq0\}$ and $\Omega^\epsilon:=\{u_{\lambda}+\epsilon\phi<0\}.$ Put
$\Theta_\epsilon:=\Omega_\epsilon\times \Omega_\epsilon.$ Since $ u_\lambda$ is a local minimizer for $J_\lambda$, replacing
$\varphi$ with $\Psi$ in \eqref{weaksolution}, one gets
\begin{align*}
0&\displaystyle\leq([u_\lambda]_{s,p}^p)^{\sigma-1}\iint_{\mathbb{R}^{2N}}\frac{|u_\lambda(x)-u_\lambda(y)|^{p-2}(u_\lambda(x)-u_\lambda(y))
 (\psi(x)-\psi(y))}{|x-y|^{N +ps}}\mathrm{d}x\mathrm{d}y\\
&- \displaystyle \lambda\int_\Omega  (u_\lambda^+)^{-\gamma}\Psi\,{\rm d}x-\int_\Omega (u_\lambda^+)^{p_s^{*}-1} \Psi
\,{\rm d}x \\
&\displaystyle=([u_\lambda]_{s,p}^p)^{\sigma-1}\iint_{\mathbb{R}^{2N}}\frac{|u_\lambda(x)-u_\lambda(y)|^{p-2}(u_\lambda(x)-u_\lambda(y))
 ((u_{\lambda}+\epsilon\phi)(x)-(u_{\lambda}+\epsilon\phi)(y))}{|x-y|^{N +ps}}\mathrm{d}x\mathrm{d}y
\\
&-\int_{\{(x,y)\in\;\Omega^\epsilon\times\Omega^\epsilon\}}\left( \lambda (u_\lambda^+)^{-\gamma}(u_{\lambda}+\epsilon\phi)
+(u_\lambda^+)^{p_s^{*}-1}(u_{\lambda}+\epsilon\phi)\right)
\,{\rm d}x\\
&\displaystyle=([u_\lambda]_{s,p}^p)^{\sigma-1}\iint_{\mathbb{R}^{2N}}\frac{|u_\lambda(x)-u_\lambda(y)|^{p-2}(u_\lambda(x)-u_\lambda(y))
 ((u_{\lambda}+\epsilon\phi)(x)-(u_{\lambda}+\epsilon\phi)(y))}{|x-y|^{N +ps}}\mathrm{d}x\mathrm{d}y\\
&\displaystyle-\int_{\Omega}\left( \lambda (u_\lambda^+)^{-\gamma}(u_{\lambda}+\epsilon\phi)
+(u_\lambda^+)^{p_s^{*}-1}(u_{\lambda}+\epsilon\phi)\right)
\,{\rm d}x\\
&\displaystyle= ([u_\lambda]_{s,p}^p)^{\sigma-1}\iint_{\mathbb{R}^{2N}}\frac{|u_\lambda(x)-u_\lambda(y)|^{p-2}(u_\lambda(x)-u_\lambda(y))
 ((u_{\lambda}+\epsilon\phi)(x)-(u_{\lambda}+\epsilon\phi)(y))}{|x-y|^{N +ps}}\mathrm{d}x\mathrm{d}y\\
&\displaystyle-\int_{\{(x,y)\in\;\Theta_\epsilon\}}\left( \lambda (u_\lambda^+)^{-\gamma}(u_{\lambda}+\epsilon\phi)
+(u_\lambda^+)^{p_s^{*}-1}(u_{\lambda}+\epsilon\phi)\right)
\,{\rm d}x\\
&\displaystyle=([u_\lambda]_{s,p}^p)^{\sigma-1}\|u_{\lambda}\|^p- \lambda\int_\Omega  (u_\lambda^+)\,{\rm d}x-\lambda\int_\Omega (u_\lambda^+)^{p_s^{*}}
\,{\rm d}x-\int_{\Omega}\left(\lambda(u_\lambda^+)^{-\gamma}\phi+(u_\lambda^+)^{p_s^{*}-1})\phi\right)\,{\rm d}x\\
&\displaystyle+ \epsilon ([u_\lambda]_{s,p}^p)^{\sigma-1}\int_{\mathbb{R}^{2N}}\frac{|u_\lambda(x)-u_\lambda(y)|^{p-2}(u_\lambda(x)-u_\lambda(y))(\phi(x)-\phi(y))}{|x-y|^{N +ps}}dxdy\\
&\displaystyle-([u_\lambda]_{s,p}^p)^{\sigma-1}\int_{\{(x,y)\in\Theta_\epsilon\}}\frac{|u_\lambda(x)-u_\lambda(y)|^{p-2}(u_\lambda(x)-u_\lambda(y))
 ((u_{\lambda}+\epsilon\phi)(x)-(u_{\lambda}+\epsilon\phi)(y))}{|x-y|^{N +ps}}\mathrm{d}x\mathrm{d}y\\
&\displaystyle-\int_{\{(x,y)\in\;\Theta_\epsilon\}}\left(\lambda  (u_{\lambda}^+)^{-\gamma}(u_{\lambda}+\epsilon\phi)
+(u_\lambda^+)^{p_s^{*}-1}(u_{\lambda}+\epsilon\phi)\right)
\,{\rm d}x\\
&\displaystyle= \epsilon ([u_\lambda]_{s,p}^p)^{\sigma-1}\int_{\mathbb{R}^{2N}}\frac{|u_\lambda(x)-u_\lambda(y)|^{p-2}(u_\lambda(x)-u_\lambda(y))(\phi(x)-\phi(y))}{|x-y|^{N +ps}}dxdy\\
&\displaystyle-\epsilon\int_{\Omega}\left(\lambda  (u_{\lambda}^+)^{-\gamma}\phi+(u_\lambda^+)^{p_s^{*}-1}\phi\right)\,{\rm d}x\\
&\displaystyle-([u_\lambda]_{s,p}^p)^{\sigma-1}\int_{\{(x,y)\in\Theta_\epsilon\}}\frac{|u_\lambda(x)-u_\lambda(y)|^{p-2}(u_\lambda(x)-u_\lambda(y))
 ((u_{\lambda}+\epsilon\phi)(x)-(u_{\lambda}+\epsilon\phi)(y))}{|x-y|^{N +ps}}\mathrm{d}x\mathrm{d}y\\
&\displaystyle-\int_{\{(x,y)\in\Theta_\epsilon\}}\left(\lambda  (u_{\lambda}^+)^{-\gamma}(u_{\lambda}+\epsilon\phi)
+(u_\lambda^+)^{p_s^{*}-1}(u_{\lambda}+\epsilon\phi)\right)
\,{\rm d}x\\
&\displaystyle\leq \epsilon ([u_\lambda]_{s,p}^p)^{\sigma-1}\int_{\mathbb{R}^{2N}}\frac{|u_\lambda(x)-u_\lambda(y)|^{p-2}(u_\lambda(x)-u_\lambda(y))(\phi(x)-\phi(y))}{|x-y|^{N +ps}}dxdy\\
&\displaystyle-\epsilon\int_{\Omega}\left(\lambda  (u_{\lambda}^+)^{-\gamma}\phi+(u_\lambda^+)^{p_s^{*}-1}\phi\right)\,{\rm d}x\\
&-([u_\lambda]_{s,p}^p)^{\sigma-1}\int_{\{(x,y)\in\Theta_\epsilon\}}\frac{|u_\lambda(x)-u_\lambda(y)|^{p-2}(u_\lambda(x)-u_\lambda(y))
 ((u_{\lambda}+\epsilon\phi)(x)-(u_{\lambda}+\epsilon\phi)(y))}{|x-y|^{N +ps}}\mathrm{d}x\mathrm{d}y,
\end{align*}
since the measure
$\Omega_\epsilon$ goes to zero as $\epsilon\to  0^{+}.$ We deduce that,
\begin{eqnarray*}
([u_\lambda]_{s,p}^p)^{\sigma-1}\int_{\{(x,y)\in\Theta_\epsilon\}}\frac{|u_\lambda(x)-u_\lambda(y)|^{p-2}(u_\lambda(x)-u_\lambda(y))
 ((u_{\lambda}+\epsilon\phi)(x)-(u_{\lambda}+\epsilon\phi)(y))}{|x-y|^{N +ps}}\mathrm{d}x\mathrm{d}y\to 0.
\end{eqnarray*}
as $\epsilon\to  0^{+}.$ We divide by $\epsilon$ and passing to the limit as
$\epsilon\to  0^{+}$, one has
\begin{eqnarray*}
&\displaystyle ([u_\lambda]_{s,p}^p)^{\sigma-1}\int_{\mathbb{R}^{2N}}\frac{|u_\lambda(x)-u_\lambda(y)|^{p-2}(u_\lambda(x)-u_\lambda(y))(\phi(x)-\phi(y))}{|x-y|^{N +ps}}dxdy\\
&\displaystyle-\int_{\Omega}\left(\lambda  (u_{\lambda}^+)^{-\gamma}\phi+(u_\lambda^+)^{p_s^{*}-1}\phi\right)\,{\rm d}x\geq0.
\end{eqnarray*}
The equality holds if we change $\phi$ by $-\phi$. So  we deduce that  $u_{\lambda}$
is a  nonnegative  solution of
 problem \eqref{eq}.
\end{proof}

\section{A perturbed problem}\label{multiplicity}

Since $J_{\lambda}$ is not Fr\'echet differentiable due to the singular term, we cannot apply the usual variational theory to the functional energy. Therefore, in order to establish
 the existence of a second solution, we  introduce the following perturbed problem
\begin{equation} \label{eqn}
\left\{
  \begin{array}{ll}
([u]_{s,p}^p)^{\sigma-1}(-\Delta)^s_p u
= \frac{\lambda}{(u^{+}+\frac{1}{n})^{\gamma}}+(u^{+})^{ p_s^{*}-1 }\quad \text{in }\Omega,\\
u=0,\;\;\;\;\quad \text{in }\mathbb{R}^{N}\setminus \Omega.
  \end{array}
\right.
\end{equation}
Associated to problem \eqref{eqn}, we consider the functional $J_{n, \lambda}: X_{0}\rightarrow \mathbb{R}$ defined by
\begin{eqnarray*}
J_{n, \lambda}(u):=\frac{1}{p\sigma} \|u\|^{p\sigma} -\frac{\lambda}{1-\gamma}\int_{\Omega}\left((u^{+}+\frac{1}{n})^{1-\gamma}-
(\frac{1}{n})^{1-\gamma}\right) \mathrm{d}x -\frac{1}{p_s^{*}}\int_{\Omega}(u^{+})^{p_s^{*}}  \mathrm{d}x.
\end{eqnarray*}
It is clear that  $J_{n, \lambda}$ is Fr\'echet differentiable, and for all $\varphi\in X_{0}$,  we have
\begin{eqnarray}\label{2.5}
&&  <J'_{n, \lambda}(u),\varphi>=  \|u\|^{p\sigma-2}<u,\varphi> - \lambda \int_{\Omega}\frac{\varphi}{\left(u^{+}+\frac{1}{n}\right)^{1-\gamma}} \mathrm{d}x - \int_{\Omega}(u^{+})^{p_s^{*}-1} \varphi \mathrm{d}x.\nonumber\\
&&
\end{eqnarray}

\begin{lemma}\label{lem1n} Let  $\rho\in(0,1]$, $\lambda_{1}$ and $\alpha$ be the constants given by Lemma \ref{lem1}. Then   for any $\lambda\in(0,\lambda_{1}]$,  one has
$$J_{n,\lambda}(u)\geq\alpha,\;\;\ \mbox{for all} \ u\in X_{0}\mbox{ with } \|u\|\leq\rho.$$
Moreover, $ \mbox{there exists} \   e\in X_{0}$, with $\|e\|>\rho$ and $J_{n,\lambda}(e)<0$.
\end{lemma}
\begin{proof}
  Since  $(u^++\frac{1}{n})^{1-\gamma}-(\frac{1}{n})^{1-\gamma}\leq (u^+)^{1-\gamma}$,  we have
$$J_{n,\lambda}(u)\geq J_{\lambda}(u).$$
Therefore, Lemma \ref{lem1} implies that the first part of Lemma \ref{lem1n} has been proved.\\
Now, let $u\in X_{0}$ with $u^+\not\equiv 0$. Then  for any $t>0$, we have
\begin{eqnarray*}
  J_{n,\lambda}(t u) &=&  \frac{t^{p\sigma}}{p} \| u\|^{p\sigma} -\frac{\lambda t^{1-\gamma}}{1-\gamma}\int_{\Omega}\left((u^{+}+\frac{1}{n})^{1-\gamma}-(\frac{1}{n})^{1-\gamma}\right) \mathrm{d}x -\frac{t^{p_s^{*}}}{p_s^{*}}\int_{\Omega}(u^{+})^{p_s^{*}}  \mathrm{d}x.
\end{eqnarray*}
Since $1-\gamma<1\leq p\sigma\leq p_s^{*}$, it follows that
 $J_{n,\lambda}(t u)\rightarrow-\infty \mbox{ as } t\rightarrow \infty.$ Hence, the  the second part of Lemma \ref{lem1n} is proved.
\end{proof}
Now, put
\begin{eqnarray}\label{clambda}
&&C_{\lambda}:=(\frac{1}{p\sigma}-\frac{1}{p_s^{*}})S^{\frac{N\sigma}{ps\sigma-N(\sigma-1)}}
-(\frac{1}{p\sigma}-\frac{1}{p_s^{*}})^{-\frac{1-\gamma}{p\sigma-1+\gamma}}
\left[\lambda
(\frac{1}{1-\gamma}+\frac{1}{p_s^{*}})|\Omega|^{\frac{p_s^{*}-1+\gamma}{p_s^{*}}}S^{-\frac{1-\gamma}{p}}\right]^{\frac{p\sigma}{p\sigma-1+\gamma}}\nonumber\\
&&
\end{eqnarray}
We show the following  useful result.
\begin{lemma} \label{lem4.2} The functional  $J_{n,\lambda}$ satisfies the (PS) condition at any level $c\in \mathbb{R}$ such that $c<C_{\lambda}$ for any $\lambda>0.$
\end{lemma}
\begin{proof}
Let 
$\{u_{k}\}\subset X_{0}$ be a (PS) minimizing  sequence for the functional  $J_{n,\lambda}$ at level $c\in \mathbb{R}$, that is
\begin{equation}\label{ps}
  J_{n,\lambda}(u_{k})\rightarrow c\;\;\;\mbox{ and } \;\; J'_{n,\lambda}(u_{k})\rightarrow 0\;\;\mbox{ as }\;\;k\rightarrow\infty.
\end{equation}
Then  by  the Sobolev embedding and the H\"older inequality, $ \mbox{there exist} \
\epsilon>0$ and $C>0$ satisfying
\begin{eqnarray*}
  c+\epsilon \|u_k\|+\circ(1)  &\geq& J_{n,\lambda}- \frac{1}{p_s^{*}}\prec J'_{n,\lambda}(u_{k}),u_{k} \succ \\
    &=& (\frac{1}{p\sigma}-\frac{1}{p_s^{*}})\|u_{k}\|^{p\sigma}   -\frac{\lambda}{1-\gamma}\int_{\Omega}\left((u_{k}^{+}+\frac{1}{n})^{1-\gamma}-
(\frac{1}{n})^{1-\gamma}\right) \mathrm{d}x\\
 &&+\frac{\lambda}{p_s^{*}}\int_{\Omega}(u_{k}^{+}+\frac{1}{n})^{-\gamma} u_{k}\mathrm{d}x\\
    &\geq&  (\frac{1}{p\sigma}-\frac{1}{p_s^{*}})\|u_{k}\|^{p\sigma}   -\lambda(\frac{1}{1-\gamma}+\frac{1}{p_s^{*}})\int_{\Omega}|u_{k}|^{1-\gamma}\mathrm{d}x \\
    &\geq&   (\frac{1}{p\sigma}-\frac{1}{p_s^{*}})\|u_{k}\|^{p\sigma}   -\lambda C (\frac{1}{1-\gamma}+\frac{1}{p_s^{*}})|\Omega|^{\frac{p_s^{*}+1-\gamma}{p_s^{*}}}\|u_{k}\|^{1-\gamma}.
\end{eqnarray*}
Since $1-\gamma<1<p\sigma<p_s^{*}$, it follows that
 $\{u_{k}\}$ is bounded. Moreover, $\{u^-_{k}\}$ is bounded in $X_0$. So from \eqref{ps}, we deduce that
$$\displaystyle\lim_{k\to\infty}\prec J'_{n, \lambda}(u_k),u_k\succ = \displaystyle\lim_{k\to\infty}\|u_k\|^{p(\sigma-1)}\prec u_k, -u^-_k\succ.$$

On the other hand, by an elementary inequality $$(a-b)(a^--b^-)\leq -(a^--b^-)^2$$
 we have
\begin{align}\label{eq-positive}
0&\leq\iint_{\mathbb{R}^{2N}} \frac{|u(x)-u(y)|^{p-2}(u(x)-u(y))
 (u^-(x)-u^-(y))}{|x-y|^{N +ps}}\mathrm{d}x\mathrm{d}y\nonumber \\
 & \leq -
\iint_{\mathbb{R}^{2N}} \frac{|u(x)-u(y)|^{p-2}
 (u^-(x)-u^-(y))^2}{|x-y|^{N +ps}}\mathrm{d}x\mathrm{d}y.
\end{align}
From \eqref{eq-positive}, we have $\|u^-_k\|\to 0$ as $k$ tends to infinity. Hence, for $k$ large enough, we have
$$J_{n,\lambda}(u_k)=J_{n,\lambda}(u^+_k)+\circ(1)\;\;\;\mbox{ and }\;\;J'_{n,\lambda}(u_k)=J'_{n,\lambda}(u^+_k)+\circ(1),$$
i.e., we can assume that $\{u_k\}$ is a sequence of nonnegative functions.\\
Now, since $\{u_k\}$ is bounded, up to a subsequence and using \cite{bario,serv}, there exist  $\{u_k\}\subset X_0$, $u$ in $X_0$, and nonnegative numbers $l, \mu$   such that
\begin{equation}\label{conv}
  \left\{
    \begin{array}{ll}
u_k \rightharpoonup u \;\;\mbox{weakly in } X_0, \\
 u_k \rightharpoonup u \;\;\mbox{weakly in } L^{p_s^{*}}(\Omega), \\
  u_k \rightarrow u \;\;\mbox{strongly in } L^{q}(\Omega)\;\;\mbox{ for }\;q\in[1, p_s^{*}),  \\
  u_k \rightarrow u \;\;\mbox{a.e. in  } \Omega,
    \end{array}
  \right.
\end{equation}
and
\begin{equation}\label{conv1}
  \left\{
    \begin{array}{ll}
\|u_k\| \rightarrow \mu, \\
 \|u_k-u\|_{p_s^{*}} \rightarrow l.
    \end{array}
  \right.
\end{equation}
Moreover, for a fixed $q\in[1, p_s^{*})$, there is $h\in L^{q}(\Omega)$ such that $$u\leq h\;\;\;\;\;\;\mbox{a.e. in  }\;\; \Omega.$$
It is easy to see that if $\mu=0$, then  $u_k\to 0$ in $X_0$. So let us assume that $\mu>0$. It follows from the above assertion that
 $$\left|\frac{u_k-u}{(u_k+\frac{1}{n})^{\gamma}}\right|\leq n^{\gamma}(h+|u|).$$
Therefore, the dominated convergence theorem implies that
\begin{equation}\label{4.6}
 \displaystyle\lim_{k\to \infty}\int_{\Omega}\frac{u_k-u}{(u_k+\frac{1}{n})^{\gamma}}dx=0.
\end{equation}
Hence, the Bresis-Lieb Lemma \cite{lieb} yields
\begin{equation}\label{4.7}
  \|u_k\|^{p}=\|u_k-u\|^{p}+\|u\|^{p}+\circ (1) \mbox{ and } \|u_k\|_{p_s^{*}}^{p_s^{*}}=\|u_k-u\|_{p_s^{*}}^{p_s^{*}}+\|u\|_{p_s^{*}}^{p_s^{*}}+\circ (1).
\end{equation}
Now, using \eqref{4.6} and \eqref{4.7}, we can deduce that:
\begin{eqnarray*}
  \circ(1) &=& \prec J'_{n,\lambda}(u_k),u_k-u\succ \\
    &=&  \|u_k\|^{p(\sigma-1)}\prec u_k,u_k-u\succ -\lambda \int_{\Omega}\frac{u_k-u}{(u_k+\frac{1}{n})^{\gamma}}dx-\int_{\Omega}u_k^{p_s^{*}-1}(u_k-u)dx\\
    &=& \mu^{p(\sigma-1)}(\|u_k\|^{p}-\|u\|^{p})-\|u_k\|_{p_s^{*}}^{p_s^{*}}+\|u\|_{p_s^{*}}^{p_s^{*}}+ \circ(1) \\
    &=&   \mu^{p(\sigma-1)}\|u_k-u\|^{p}-\|u_k-u\|_{p_s^{*}}^{p_s^{*}}+ \circ(1).
\end{eqnarray*}
Therefore,
\begin{equation}\label{4.8}
\mu^{p(\sigma-1)}\displaystyle\lim_{k\rightarrow \infty}\|u_k-u\|^{p}=\displaystyle\lim_{k\rightarrow \infty}\|u_k-u\|_{p_s^{*}}^{p_s^{*}}=l.
\end{equation}
Since $\mu>0$, if $l=0$, we obtain that $u_k\rightarrow u$ in $X_0$ and the proof is complete.

 Now, let us prove that $l=0$. Proceeding  by contradiction, suppose that $l>0$. Then  from  \eqref{4.8} and the Sobolev embedding, we get
\begin{equation}\label{4.9}
  S\mu^{p(\sigma-1)}l^{p}\leq l^{p_s^{*}},
\end{equation}
that is
\begin{equation}\label{4.99}
  l^{p_s^{*}-p}\geq S \mu^{p(\sigma-1)}.
\end{equation}
On the other hand, by combining \eqref{4.7} and \eqref{4.8}, we obtain
$$
    \mu^{p(\sigma-1)}(\mu^p-\|u\|^{p})=l^{p_s^{*}}
$$
that is,
$$
   l= \mu^{\frac{p(\sigma-1)}{p_s^{*}}}(\mu^p-\|u\|^{p})^{\frac{N-ps}{Np}}.
$$
So using \eqref{4.99}, we get
$$
   l^{p_s^{*}-p}= \mu^{\frac{p(p_s^{*}-p)(\sigma-1)}{p_s^{*}}}(\mu^p-\|u\|^{p})^{\frac{(N-ps)(p_s^{*}-p)}{Np}}\geq S \mu^{p(\sigma-1)}l^{p}.
$$
We deduce that
$$
   \mu^{\frac{p^{2}s}{N}}\geq (\mu^p-\|u\|^{p})^{\frac{(N-ps)(p_s^{*}-p)}{Np}}\geq S \left(\mu^{p(\sigma-1)}\right)^{\frac{N-ps}{N}}.
$$
Since  $1<\sigma<\frac{p_s^{*}}{p}$, it follows that $ps\sigma-N(\sigma-1)>0$. So
\begin{equation}\label{4.10}
  \mu^{p}\geq S^{\frac{N}{ps\sigma-N(\sigma-1)}}.
\end{equation}
Now, the fact that $(u^++\frac{1}{n})^{1-\gamma}-(\frac{1}{n})^{1-\gamma}\leq (u^+)^{1-\gamma}$ implies that for all integers $k$ and $n$
we have
\begin{eqnarray*}
J_{n,\lambda}(u_k)-\frac{1}{p_s^{*}}\prec J'_{n,\lambda}(u_k),u_k\succ \geq (\frac{1}{p\sigma}-\frac{1}{p_s^{*}})\|u_k\|^{p\sigma}-\lambda
(\frac{1}{1-\gamma}+\frac{1}{p_s^{*}})\int_{\Omega}u_k^{1-\gamma}\;dx.
\end{eqnarray*}
So from  \eqref{4.7}, \eqref{4.10}, the H\"older inequality and the
Young inequality, if $k$ tends to infinity, we get
\begin{eqnarray*}
  c &\geq&  (\frac{1}{p\sigma}-\frac{1}{p_s^{*}})(\mu^{p\sigma}+\|u\|^{p\sigma}) -\lambda
(\frac{1}{1-\gamma}+\frac{1}{p_s^{*}})|\Omega|^{\frac{p_s^{*}-1+\gamma}{p_s^{*}}}S^{-\frac{1-\gamma}{p}}\|u\|^{1-\gamma} \\
    &\geq&  (\frac{1}{p\sigma}-\frac{1}{p_s^{*}})(\mu^{p\sigma}+\|u\|^{p\sigma})-(\frac{1}{p\sigma}-\frac{1}{p_s^{*}}) \|u\|^{p\sigma}\\
& & -(\frac{1}{p\sigma}-\frac{1}{p_s^{*}})^{-\frac{1-\gamma}{p\sigma-1+\gamma}}
\left[\lambda
(\frac{1}{1-\gamma}+\frac{1}{p_s^{*}})|\Omega|^{\frac{p_s^{*}-1+\gamma}{p_s^{*}}}S^{-\frac{1-\gamma}{p}}\right]^{\frac{p\sigma}{p\sigma-1+\gamma}}\\
    &\geq&   (\frac{1}{p\sigma}-\frac{1}{p_s^{*}})S^{\frac{N\sigma}{ps\sigma-N(\sigma-1)}}-(\frac{1}{p\sigma}-\frac{1}{p_s^{*}})^{-\frac{1-\gamma}{p\sigma-1+\gamma}}
\left[\lambda
(\frac{1}{1-\gamma}+\frac{1}{p_s^{*}})|\Omega|^{\frac{p_s^{*}-1+\gamma}{p_s^{*}}}S^{-\frac{1-\gamma}{p}}\right]^{\frac{p\sigma}{p\sigma-1+\gamma}}\\
&=&C_{\lambda},
\end{eqnarray*}
which is a contradiction.
\end{proof}

\section{Existence of an upper bound}
Under some   suitable condition, we shall prove that $J_{n,\lambda}$ is bounded   from above. To this end, we can
assume without loss of generality, that $0\in \Omega$ and we fix $r>0$ such that $B_{4r}\subset \Omega$ where $B_{4r}:=\{x\in\mathbb{R}^{N}: |x|<4r\}$. Let $\varepsilon>0$ and $\psi_{\epsilon}$ be the function defined by
\begin{equation}\label{psi}
\psi_{\epsilon}:=\frac{\phi U_{\epsilon}}{\|\phi U_{\epsilon}\|_{p_s^{*}}^{p}},
\end{equation}
where  $U_{\epsilon}$ is the family of functions (for more details see \cite{PeSqYa}) and $\phi\in C^{\infty}(\mathbb{R}^{N}, [0,1])$ is satisfying
$$\phi=\left\{
         \begin{array}{ll}
           1\;\;\mbox{in}\;\;B_{r}, \\
           0\;\;\mbox{in}\;\;\mathbb{R}^{N}\backslash B_{2r}
         \end{array}
       \right.$$
\begin{lemma}\label{pals} $\ \mbox{There exist} \   \lambda_{2}>0$ and $\psi\in E$ satisfying
$$\displaystyle\sup_{t>0}J_{n,\lambda}(t\psi)<C_\lambda,$$
for all $\lambda\in(0,\lambda_1).$
\end{lemma}
\begin{proof}
Let $\epsilon>0$ and let $u_{\epsilon}$ and $\psi_{\epsilon}$ be as above. Since $$0<1-\gamma<p\sigma<p_s^{*},$$
  it is easy to see that
$$J_{n,\lambda}(t\psi_{\epsilon})\longrightarrow -\infty\;\;\mbox{ as }\;\;t\rightarrow \infty.$$
Thus, $\ \mbox{there exists} \
t_{\epsilon}>0$ satisfying
$$J_{n,\lambda}(t_{\epsilon}\psi_{\epsilon})=\displaystyle\max_{t\geq0}J_{n,\lambda}(t\psi_{\epsilon}).$$
From Lemma \ref{lem1},  we get $J_{n,\lambda}\geq\alpha>0$. So since the functional $J_{n,\lambda}$ is continuous, we deduce the existence of  two values $t_{0}, t^{\ast} > 0$ satisfying
$$t_{0}<t_\epsilon< t_{1},\;\;\;\mbox{and }\;\;\;J_{n,\lambda}(t_{0}\psi_{\epsilon})=J_{n,\lambda}(t_{1}\psi_{\epsilon})=0.$$
On the other hand, since $\|u_{\epsilon}\|_{p_s^{*}}$ is independent from $\epsilon,$ it follows  from \cite{mosc} that
$$\|\psi_{\epsilon}\|^{p}\leq\frac{\iint_{\mathbb{R}^{2N}}\frac{|u(x)-u(y)|^p}{|x-y|^{N+ps}}\,dx\,dy}
{\|\phi{u_{\epsilon}\|^p_{p_s^{*}}}}=S+O(\epsilon^{\frac{N-ps}{p-1}}).$$
In fact, for any $a>0, b\in[0,1], p\geq1,$ $$(a+b)^p\leq a^p+p(a+1)^{p-1}b.$$
We obtain  for $\epsilon$ small enough,
$$ \|\psi_{\epsilon}\|^{p\sigma}\leq (S+O(\epsilon^{\frac{N-ps}{p-1}}))^{\sigma}\leq S^\sigma +O(\epsilon^{\frac{N-ps}{p-1}}).$$
Hence, for any $\epsilon>0$ sufficiently small, and using the fact that $t_{0}<t_\epsilon< t_{1}$ and $\|\psi_{\epsilon}\|^{p^{\ast}\sigma}=1,$ we obtain
\begin{eqnarray}\label{4.14}
J_{n,\lambda}(t_{\epsilon} \psi_{\epsilon})&\leq& \left( \frac{t_{\epsilon}^{p\sigma}}{\sigma p}S^{\sigma}-\frac{t_{\epsilon}^{p_s^{*}}}{p_s^{*}}\right) -\frac{\lambda }{1-\gamma}\int_{\Omega}\left((t_{0} \psi_{\epsilon}+\frac{1}{n})^{1-\gamma}-(\frac{1}{n})^{1-\gamma}\right) \mathrm{d}x
\nonumber\\
& +&O(\epsilon^{\frac{N-ps}{p-1}}).
\end{eqnarray}
Since
\begin{equation}\label{max}
 \displaystyle\max_{t>0}\left( \frac{t^{p\sigma}}{\sigma p}S^{\sigma}-\frac{t^{p_s^{*}}}{p_s^{*}}\right)=\left( \frac{1}{\sigma p}-\frac{1}{p_s^{*}}\right)S^{\frac{p_s^{*} \sigma}{p_s^{*}-p\sigma} },
\end{equation}
it follows by \eqref{4.14} and \eqref{max} that
\begin{eqnarray}\label{4.15}
J_{n,\lambda}(t_{\epsilon} \psi_{\epsilon}) &\leq& \left( \frac{1}{\sigma p}-\frac{1}{p_s^{*}}\right)S^{\frac{p_s^{*} \sigma}{p_s^{*}-p\sigma} } -\frac{\lambda }{1-\gamma}\int_{\Omega}\left((t_{0} \psi_{\epsilon}+\frac{1}{n})^{1-\gamma}-(\frac{1}{n})^{1-\gamma}\right) \mathrm{d}x\nonumber\\ &+&O(\epsilon^{\frac{N-ps}{p-1}}).
\end{eqnarray}
In addition, for any $a>0, \;b>0$ large enough,
$$(a+b)^{\epsilon}-a^{\epsilon}\geq \epsilon b^{\frac{\epsilon}{p}}a^{\frac{\epsilon(p-1)}{p}}.$$
We can now deduce that for all $q>0$  small enough, we can establish the existence of  $c_1>0$ satisfying
\begin{align*}
  &\int_{\Omega}\left((t_{0} \psi_{\epsilon}+\frac{1}{n})^{1-\gamma}-(\frac{1}{n})^{1-\gamma}\right) \mathrm{d}x \\&\geq
\displaystyle c_1(1-\gamma)\epsilon^{\frac{(N-ps)(1-\gamma)}{p_s^{*}}}\int_{x\in\Omega:|x|\leq \epsilon^{q}}\left(\frac{1}{(\frac{1}{|x|^{p'}+\epsilon^{p'}})^{\frac{N-ps}{p}}}\right)^{\frac{p(1-\gamma)}{p_s^{*}}}dx \\
    &\geq \displaystyle c_1(1-\gamma) \epsilon^{\frac{(N-ps)(1-\gamma)-p(p-1)q(N-ps)(1-\gamma)+p_s^{*}qN}{p_s^{*}}}.
\end{align*}
Combining this with \eqref{4.15}, we get
\begin{eqnarray}\label{4.18}
  J_{n,\lambda}(t_{\epsilon} \psi_{\epsilon}) &\leq& \displaystyle \left( \frac{1}{\sigma p}-\frac{1}{p_s^{*}}\right)S^{\frac{p_s^{*} \sigma}{p_s^{*}-p\sigma} } - \lambda c_1  \epsilon^{\frac{(N-ps)(1-\gamma)-p(p-1)q(N-ps)(1-\gamma)+p_s^{*}qN}{p_s^{*}}} +O(\epsilon^{\frac{N-ps}{p-1}})\nonumber \\
    &\leq& \displaystyle( \frac{1}{\sigma p}-\frac{1}{p_s^{*}})S^{\frac{p_s^{*} \sigma}{p_s^{*}-p\sigma} } - \lambda  c_1 \epsilon^{\frac{(N-ps)(1-\gamma)-p(p-1)q(N-ps)(1-\gamma)+p_s^{*}qN}{p_s^{*}}} +c_2 \epsilon^{\frac{N-ps}{p-1}},\nonumber\\
    & &
\end{eqnarray}
for some positive constant $c_2$.\\
Now, let $\widetilde{\lambda}>0$ be such that $C_{\lambda}>0$ for all $\lambda\in (0,\widetilde{\lambda})$, where $C_{\lambda}$ is given by \eqref{clambda} and let us set
\begin{eqnarray*}
\beta:=1+\frac{p(p-1)\sigma\left((N-ps)(1-\gamma)-p(p-1)q(N-ps)(1-\gamma)+p_s^{*}qN\right)}{p_s^{*}
(p\sigma-1+\gamma)(N-ps)}-\frac{p\sigma}{p\sigma-1+\gamma},
\end{eqnarray*}
$$\theta:=(\frac{1}{p\sigma}-\frac{1}{p_s^{*}})^{-\frac{1-\gamma}{p\sigma-1+\gamma}}
\left[
(\frac{1}{1-\gamma}+\frac{1}{p_s^{*}})|\Omega|^{\frac{p_s^{*}-1+\gamma}{p_s^{*}}}S^{-\frac{1-\gamma}{p}}\right]^{\frac{p\sigma}{p\sigma-1+\gamma}},$$
and $$\lambda_2:=\min\left\{\widetilde{\lambda}, r^{\frac{(p\sigma-1+\gamma)(N-ps)}{pq\sigma}}, \left(\frac{c_2+\theta}
{c_{1}}\right)^{\frac{1}{\beta}}\right\},$$
where  $r>0$ is such that $B_{4r}\subset \Omega$ and $q>0$ is such that $\beta<0$.\\
Now, for $\lambda\in(0,\lambda_1)$, if we choose $$\epsilon:=\lambda^{\frac{p(p-1)\sigma}{(ps-1+\gamma)(N-ps)}}$$ in \eqref{4.18}. Then  using the fact that $c_1\lambda^{\beta}>c_1\lambda_1^{\beta}\geq c_2+\theta$, we obtain
\begin{eqnarray*}
  J_{n,\lambda}(t_{\epsilon} \psi_{\epsilon}) &\leq& ( \frac{1}{\sigma p}-\frac{1}{p_s^{*}})S^{\frac{p_s^{*} \sigma}{p_s^{*}-p\sigma} } - \lambda  c_1 \lambda^{\beta+\frac{p\sigma}{ps-1+\gamma}-1} +c_2 \lambda^{\frac{p\sigma}{ps-1+\gamma}} \\
    &=& ( \frac{1}{\sigma p}-\frac{1}{p_s^{*}})S^{\frac{p_s^{*} \sigma}{p_s^{*}-p\sigma}}+\lambda^{\frac{p\sigma}{p\sigma-1+\gamma}}(c_2-c_1\lambda^{\beta}) \\
    &<& ( \frac{1}{\sigma p}-\frac{1}{p_s^{*}})S^{\frac{p_s^{*} \sigma}{p_s^{*}-p\sigma} }-\theta\lambda^{\frac{p\sigma}{p\sigma-1+\gamma}}=C_{\lambda}.
\end{eqnarray*}
\end{proof}
Set $$\lambda_{0}:=\min(\lambda_{1}, \lambda_{2}).$$
Then we have the following important result.
\begin{lemma}\label{fin} Problem \eqref{eqn} has a nonnegative solution $v_n\in X_0$ satisfying
$$\alpha<J_{n, \lambda}(v_n)<C_\lambda,$$
for all $\lambda\in(0,\lambda_0),$ where $\alpha$ is from Lemma \ref{lem1}.
\end{lemma}
\begin{proof} Let $\lambda\in(0,\lambda_0)$. By Lemma \ref{lem1},   $J_{n, \lambda}$ satisfies the Mountain Pass geometry. So we can define  the Mountain Pass level
$$c_{n, \lambda}:=\displaystyle\inf_{g\in \Gamma}\displaystyle\max_{t\in[0,1]}J_{n, \lambda}(g(t)),$$
where $$\Gamma:=\left\{g\in C([0,1], E):g(0)=0, J_{n, \lambda}(g(1))<0\right\}.$$
Moreover, $$0<\alpha<c_{n, \lambda}\leq \displaystyle\sup_{t\geq}J_{n, \lambda}(t\psi)<C_{n, \lambda}.$$
Hence, by Lemma \ref{lem4.2}, $J_{n, \lambda}$ satisfies the (PS) condition at the level $c_{n, \lambda}$, i.e., there exists a non-regular point $v_n$ for $J_{n, \lambda}$ at level $c_{n, \lambda}$. Moreover, $J_{n, \lambda}(v_n)=c_{n, \lambda}>\alpha>0.$ We can therefore deduce that $v_n$ is a nontrivial critical point of the functional energy $J_{n, \lambda}$ and also a solution to   problem \eqref{eqn}. If we now replace $\varphi$ by $v_{n}^-$ in \eqref{2.5} and use \eqref{eq-positive}, we get  $\|v_n\|=0$, that is, $v_n$ is nonnegative. This leads to the  positivity of $v_n$ by  the maximum principle \cite{BrFr}.
\end{proof}
\section{Proof of Theorem \ref{thm}}\label{section4}
In order to complete the proof of our main result it now remains to obtain a second positive solution to problem \eqref{eq} as a limit of the some subsequence of $|{v_n|}$. To this end, let $\lambda\in(0,\lambda_0)$ and  $|{v_n|}$ be a family of the  positive function given by Lemma \ref{fin}. By Lemma \ref{fin},   the  H\"older inequality and since $(v_n +\frac{1}{n})^{1-\gamma}-(\frac{1}{n})^{1-\gamma}\leq v_n^{1-\gamma}$, we see that
\begin{eqnarray*}
  C_\lambda &>& J_{n, \lambda}-\frac{1}{p_s^\ast}<J'_{n, \lambda}(v_n),v_n> \\
    &=&  (\frac{1}{p\sigma}-\frac{1}{p_s^{*}})\|v_{n}\|^{p\sigma}   -\frac{\lambda}{1-\gamma}\int_{\Omega}\left((v_{n} +\frac{1}{n})^{1-\gamma}-
(\frac{1}{n})^{1-\gamma}\right) \mathrm{d}x+\frac{\lambda}{p_s^{*}}\int_{\Omega}(v_{n} +\frac{1}{n})^{-\gamma} v_{n}\mathrm{d}x \\
   &\geq&   (\frac{1}{p\sigma}-\frac{1}{p_s^{*}})\|v_{n}\|^{p\sigma}   -\frac{\lambda}{1-\gamma}\int_{\Omega} v_{n}^{1-\gamma}\mathrm{d}x\\
    &\geq&   (\frac{1}{p\sigma}-\frac{1}{p_s^{*}})\|v_{n}\|^{p\sigma}   -\frac{\lambda}{1-\gamma}|\Omega|^{\frac{p_s^{*}-1+\gamma}{p_s^{*}}}S^{-\frac{1-\gamma}{p}}\|v_{n}\|^{1-\gamma}.
\end{eqnarray*}
Since $0<1-\gamma<1<p\sigma$, $v_{n}$ is bounded in $X_0.$ So,  $\ \mbox{there is} \
v_{\lambda}\in X_0$ satisfying
$$\left\{
  \begin{array}{ll}
     v_n\rightharpoonup v_\lambda\; \mbox{   weakly in } X_0, \\
      v_n\rightharpoonup v_\lambda\; \mbox{  weakly in } L^{p_s^{*}}(\Omega),\\
     v_n\rightarrow v_\lambda \;\mbox{  strongly in } L^{r}(\Omega), \mbox{ for any } r\in[1,p_s^{*}) \\
     v_n\rightarrow v_\lambda\;\;\mbox{ a.e. in } \Omega.
  \end{array}
\right.$$
We shall now prove that $v_n\rightarrow v_\lambda\; \mbox{   strongly in } X_0,$ i.e. $\|v_n-v_\lambda\|\rightarrow 0$ as $n\rightarrow\infty$.\\
 First, we observe that if  $\|v_n\|\rightarrow 0$, then $v_n\rightarrow v_\lambda\; \mbox{   strongly in } X_0$, so we assume that $\|v_n\|\rightarrow \eta>0$. Since $$0\leq\frac{v_n}{(v_n+\frac{1}{n})^{\gamma}}\leq v_n^{1-\gamma}\;\;\mbox{a.e. in }\;\Omega,$$
it follows by the Vitali theorem that
$$\displaystyle\lim_{n\rightarrow \infty}\int_{\Omega}\frac{v_n}{(v_n+\frac{1}{n})^{\gamma}}\mathrm{d}x=\int_{\Omega}v_\lambda^{1-\gamma}\mathrm{d}x.$$
Now, replace both $u$ and $\varphi$  by $v_n$ in \eqref{2.5} to get
\begin{equation}\label{5.3}
  \eta^{p\sigma}-\lambda \int_{\Omega}v_\lambda^{1-\gamma}\mathrm{d}x + \|v_n\|_{p_s^{*}}^{p_s^{*}}\rightarrow 0.
\end{equation}
On the other hand, by a  simple calculation in \eqref{eqn} we get
$$\|v_n\|^{p\sigma}(-\Delta)_{p}^{s}v_n\geq \min(1,\frac{\lambda}{p^{\gamma}})\;\;\;\mbox{in }\Omega,$$
since $v_n$ is bounded in $X_0$. Now, by the strong  maximum principle \cite{BrFr}, there exist $\widetilde{\Omega}\subset \Omega$ and $\widetilde{c}>0$ such that
\begin{equation}\label{5.4}
  v_n\geq \widetilde{c}>0,\;\;\mbox{a.e. in }\;\Omega,
\end{equation}
for any integer $n.$ Let $\varphi \in C_0^\infty(\Omega)$ satisfy $\supp(\varphi)=\widetilde{\Omega}\subset \Omega.$ Then by \eqref{5.4},
$$0\leq \left|\frac{\varphi}{(v_n+\frac{1}{n})^{\gamma}}\right|\leq \frac{|\varphi|}{\widetilde{c}},\;\;\mbox{a.e. in }\;\Omega.$$
Then the dominated convergence theorem implies that
$$
  \displaystyle\lim_{n\rightarrow \infty}\int_{\Omega}\frac{\varphi}{(v_n+\frac{1}{n})^{\gamma}}\mathrm{d}x=\int_{\Omega}v_\lambda^{-\gamma}\varphi\mathrm{d}x.
$$
Thus, by replacing $u$ with $v_\lambda$ in \eqref{2.5} and by letting $n$  to infinity, we obtain
\begin{equation}\label{5.6}
 \eta^{p(\sigma-1)}<v_\lambda,\varphi>-\lambda\int_{\Omega}v_\lambda^{-\gamma}\varphi\mathrm{d}x+\int_{\Omega}v_\lambda^{p_s^\ast-1}\varphi\mathrm{d}x=0.
\end{equation}
Now, if we replace $\varphi$ by $v_\lambda$ in \eqref{5.6} and invoke  \eqref{2.5},  we obtain
 $$=\eta^{p(\sigma-1)}(\eta^{p}-\|v_\lambda\|^{p})\displaystyle\lim_{n\rightarrow\infty}\left(\|v_n\|_{p_s^\ast}^{p_s^\ast}-\|v_\lambda\|_{p_s^\ast}^{p_s^\ast}\right).$$
Therefore, by the Brezis-Lieb Lemma \cite{lieb}, we obtain
\begin{equation}\label{5.7}
  \eta^{p(\sigma-1)}\displaystyle\lim_{n\rightarrow \infty}\left(\|v_n-v_\lambda\|^{p}\right)=l^{p_s^\ast}.
\end{equation}
Now, let us prove that $l=0$, by contradiction, i.e. we assume that $l>0$. As in Lemma \ref{lem4.2} we can prove that
$$l^{p_s^{*}-p}\geq S \mu^{p(\sigma-1)}.$$
Therefore, by Lemma \ref{fin} combined with  Young inequality and H\"older inequality, we deduce
\begin{eqnarray*}
  C_{\lambda} &>& J_{n,\lambda}(v_n)-\frac{1}{p_s^\ast}<J'_{n,\lambda}(v_n),v_n> \\
    &\geq&  \left(\frac{1}{p\sigma}-\frac{1}{p_s^\ast}\right)\left(\eta^{p\sigma}+\|v_\lambda\|^{p\sigma}\right)-\lambda \left(\frac{1}{1-\gamma}+\frac{1}{p_s^\ast}\right)|\Omega|^{\frac{p_{s}^{\ast}-1+\gamma}{p_{s}^{\ast}}}S^{-\frac{1-\gamma}{p}}\|v_\lambda\|^{1-\gamma}\\
    &\geq&  C_{\lambda}.
\end{eqnarray*}
Clearly, this is a contradiction,
so    $l=0$ and $v_n \rightarrow v_\lambda\; \mbox{   strongly in } X_0$.
In addition, one can easily see that $v_\lambda$ is a solution of
 problem \eqref{eq}.
 Therefore by Lemma \ref{fin},
  $J_{\lambda}(v_{\lambda})\geq \alpha>0$ so $v_\lambda$ is nontrivial.
  We can now proceed as in the proof of Lemma \ref{fin} and deduce that $v_\lambda$ is a positive solution of problem \eqref{eq}. In conclusion, since $J_{\lambda}(u_{\lambda})<0<J_{\lambda}(v_{\lambda}),$ this  completes the proof.\qed
\section*{Acknowledgements}
The fourth author was supported by the Slovenian Research Agency program P1-0292 and grants N1-0114 and N1-0083.

\end{document}